\documentclass[12pt, a4paper, leqno]{amsart}
\usepackage[T1]{fontenc}
\usepackage{hyperref}
\usepackage[ansinew]{inputenc} 
\usepackage[dvips,pdftex]{graphicx} 
\usepackage{amsmath,amssymb,amsthm,amsfonts,amstext}
\usepackage{xspace}
\usepackage{color}
\usepackage{newlfont}
\usepackage{longtable}
\usepackage[english]{babel}

\newcommand{\nc}{\newcommand}

\nc{\GG}{\mathfrak{G}}
\nc{\Z}{\mathbb{Z}}

\theoremstyle{plain}
\newtheorem{theorem}{\sc Theorem}[section]
\newtheorem{thm}[theorem]{\sc Theorem}
\newtheorem{lem}[theorem]{\sc Lemma}
\newtheorem{prop}[theorem]{\sc Proposition} 
\newtheorem{cor}[theorem]{\sc Corollary}

\newtheorem{rem}[theorem]{\sc Remark}

\theoremstyle{definition}
\newtheorem{define}[theorem]{\sc Definition}

\newcommand{\vfi}{\varphi}
\newcommand{\fa}{\mathfrak{a}} 
\newcommand{\fb}{\mathfrak{a}^{\vfi}}

\nc{\bea}{\begin{eqnarray}} \nc{\eea}{\end{eqnarray}}
\nc{\beae}{\begin{eqnarray*}} \nc{\eeae}{\end{eqnarray*}}

\nc{\gm}{\gamma}    
\nc{\gmm}{\raisebox{.4ex}{$\! \gamma \!$}} \nc{\st}{\star}

\nc{\la}{\leftarrow} 
\nc{\Ra}{\Rightarrow}
\nc{\La}{\Leftarrow}                
\nc{\ra}{\rightarrow}
\nc{\lla}{\longleftarrow}            
\nc{\lra}{\longrightarrow}
\nc{\Lla}{\Longleftarro     
\nc{\vfi}{\varphi}
\nc{\gmm}{\raisebox{.4ex}{$\! \gamma \!$}} \nc{\st}{\star}

\nc{\F}{\mathbb{F}}
\nc{\N}{\mathbb{N}}

\nc{\Z}{\mathbb{Z}} 

\nc{\Q}{\mathbb{Q}} 
\nc{\LL}{\mathcal{L}} 
\nc{\JJ}{\mathcal{J}}

\newcommand{\GG}{\mathfrak{G}} 

\nc{\NN}{\mathfrak{N}} 
\nc{\HH}{\mathfrak{H}}  
\newcommand{\cH}{\mathcal{H}}  
\newcommand{\cG}{\mathcal{G}}  
\nc{\cN}{\mathcal{N}}  
\nc{\fa}{\mathfrak{a}}w} 
\nc{\cH}{\mathcal{H}}  
\nc{\cG}{\mathcal{G}}  
\nc{\cN}{\mathcal{N}}  

\nc{\pgs}{\mathfrak{p}\mathfrak{g}\mathfrak{s}}
\nc{\pcr}{\mathfrak{p}\mathfrak{c}\mathfrak{r}}
\nc{\wg}{\widehat{g}}
\nc{\wh}{\widehat{h}}
\nc{\wx}{\widehat{x}}
\nc{\wy}{\widehat{y}}

\nc{\Char}{\operatorname{char}}
\nc{\Ker}{\operatorname{Ker}} 
\nc{\Imm}{\operatorname{Im}}
\nc{\Aut}{\operatorname{Aut}} 
\nc{\Cl}{\operatorname{Cl}}
\nc{\Orb}{\operatorname{Orb}} 
\nc{\noreq}{\trianglelefteq}

\nc{\bca}{\begin{cases}}                   
\nc{\eca}{\end{cases}}
\nc{\barr}{\begin{array}}                  
\nc{\earr}{\end{array}}
\nc{\bthm}{\begin{thm}}                
\nc{\ethm}{\end{thm}}
\nc{\bprop}{\begin{prop}}                      
\nc{\eprop}{\end{prop}}
\nc{\blem}{\begin{lem}}              
\nc{\elem}{\end{lem}}
\nc{\bins}{\begin{ins}}                      
\nc{\eins}{\end{ins}}
\nc{\bcor}{\begin{cor}}                      
\nc{\ecor}{\end{cor}}
\nc{\brem}{\begin{rem}}                     
\nc{\erem}{\end{rem}}
\nc{\bdeff}{\begin{define}}                        
\nc{\edeff}{\end{define}}

\advance\textheight by \topskip 
\title[q-tensor square of nilpotent groups]{The q-tensor square of finitely generated nilpotent groups, 
$q \geq 0$}
\author[Rocco]{Nora\'i R. Rocco}
\address{Departamento de Matemat\'atica, Universidade de Bras\'ilia, Brasilia-DF, 70910-900 Brazil 
}
\email{norai@unb.br}
\author[Rodrigues]{Eunice C. P. Rodrigues}
\address{Departamento de Matem\'atica, Universidade Fe\-de\-ral de Ma\-to Gros\-so, Ron\-do\-n\'o\-po\-lis-MT, 
85735-001 Brazil }
\email{eunicepr@hotmail.com}
\date{\today}

\begin{document}

\begin{abstract}
In the present paper the authors extend to the $q-$tensor square 
$G \otimes^q G$ of a group $G$, $q$ a non-negative 
integer, some structural results due to R. D. Blyth, F. Fumagalli and M. Morigi 
concerning the non-abelian tensor square $G \otimes G$ ($q = 0$). 
The results are applied to the computation of 
$G \otimes^q G$ for finitely generated nilpotent 
groups $G$, specially for free nilpotent groups of finite rank. 
We also generalize to all $q \geq 0$ results of M. Bacon 
regarding an upper bound to the minimal number of generators 
of the non-abelian tensor square $G \otimes G$ when $G$ is a 
$n-$generator nilpotent group of class 2. 
We end by computing the $q-$tensor squares of the free $n-$generator 
nilpotent group of class 2, $n \geq 2$, for all $q \geq 0.$  
This shows that the above mentioned upper bound is also achieved for these groups when $q > 1.$ 
\end{abstract}

\maketitle

\section{Introduction}

Let $G$ and $G^{\varphi}$ be groups, isomorphic via 
$\varphi: g \mapsto g^{\varphi}$ for all $g \in G$. 
Consider the group $\nu(G),$ introduced in \cite{Rocco1} as 
\begin{equation} \label{eq:presenta0}
 \nu(G) = \left \langle G \cup G^{\varphi} \, | \, 
[g, h^{\varphi}]^k = 
 [g^k, (h^k)^{\varphi}] = [g, h^{\varphi}]^{k^{\varphi}}, \, \forall g,h,k 
 \in G \right \rangle.   
\end{equation}
It is a well known fact (see \cite{Rocco1}) that the subgroup $\Upsilon(G) = [G, G^{\varphi}]$ of $\nu(G)$ 
is isomorphic to the non-abelian tensor square 
$G \otimes G$, as defined by Brown and Loday 
in their seminal paper \cite{BL}. 
A modular version of the operator $\nu$ was considered in \cite{BR}, 
where for any non-negative integer $q$ the authors introduced and studied a 
group $\nu^{q}(G),$ which in turn is an extension of the so called q-tensor 
square of $G$, $G \otimes^q G,$ first defined by Conduch\'e and 
Rodriguez-Fernandez in \cite{Cond} 
(see also \cite{Ellis}, \cite{Brown}). 
In order to describe the group $\nu^q(G),$ if $q \geq 1$ then let 
$\widehat{\mathcal{G}} = \{\widehat{k} \, | \, k \in G \}$ be a set of symbols, one for each 
element of $G$ (for $q = 0$ we set $\widehat{\mathcal{G}} = \emptyset,$ the empty set). 
Let $F(\widehat{\mathcal{G}})$ be the free group over $\widehat{\mathcal{G}}$ and 
$\nu(G) \star F(\widehat{\mathcal{G}})$ 
be the free product of $\nu(G)$ and $F(\widehat{\mathcal{G}}).$ As $G$ and $G^{\varphi}$ 
are embedded into $\nu(G)$ we shall identify the elements of  $G$ (respectively of $G^{\varphi}$) 
with their respective images in $\nu(G) \ast F(\widehat{\mathcal{G}})$.
Let $J$ denote the normal closure in $\nu(G) \ast F(\widehat{\mathcal{G}})$ 
of the following elements, for all $\widehat{k}, \widehat{k_1} \in 
\widehat{\mathcal{G}}$ and $\ g, \, h  \in G:$
 \begin{gather} 
{g}^{-1} \, \widehat{k} \, g \; \widehat{(k^g)}^{-1}; \label{RR1} \\
(g^{\varphi})^{-1} \, \widehat{k} \, g^{\varphi} \; \widehat{(k^g)}^{-1}; 
\label{RR2} \\
(\widehat{k})^{-1} [g ,h^{\varphi}] \, \widehat{k} \; [g^{k^q}, (h^{k^{q}})^{\varphi}]^{-1}; 
\label{RR3} \\
(\widehat{k})^{-1} \, \widehat{k k_1} \; (\widehat{k_1})^{-1}
 \displaystyle({\prod_{i=1}^{q-1}}[k, (k^{-i}_1)^{\varphi}]^{k^{q-1-i}})^{-1};  \label{RR4} \\
 [\widehat{k}, \widehat{k_1}] \; [k^q, (k^{q}_1)^{\varphi}]^{-1}; \label{RR5} \\
 \widehat{[g, h]} \; [g, h^{\varphi}]^{-q}.  \label{RR6}
 \end{gather} 

\begin{define} \label{nu^q}
The group $\nu^q(G)$ is defined to be the factor group 
\begin{equation} 
\nu^q(G) := ( \nu(G) \ast F(\widehat{\mathcal{G}}) )/J.
\end{equation}
\end{define}
Note that for $q=0$ the sets of relations $(\ref{RR1})$ to
$(\ref{RR6})$ are empty; in this case we have $\nu^{0}(G) = \nu(G) \ast 
F(\widehat{\mathcal{G}}))/J \cong \nu(G)$.

Let $R_1, \ldots, R_6$ be the sets of relations corresponding to
$(\ref{RR1}), \ldots, (\ref{RR6})$, respectively, and let $R$ be their union, 
$R= \bigcup^6_{i=1}R_i$. Therefore, $\nu^q(G)$ has the presentation:
 \begin{equation*} \label{presenta1}
\nu^q(G)= \left \langle G, G^{\varphi}, \widehat{\mathcal{G}} \; | \; R, 
[g,h^{\varphi}]^k \, [g^k, (h^k)^{\varphi}]^{-1}, 
[g, h^{\varphi}]^{k^{\varphi}} \, [g^k, (h^k)^{\varphi}]^{-1}, \, \forall g, h, 
k \in G \right \rangle.
 \end{equation*}

\vskip -5pt
There is an epimorphism $\rho: \eta^q(G) \twoheadrightarrow G,  g \mapsto  g, h^{\varphi} \mapsto h,
 \widehat{k} \mapsto k^q$. On the other hand the inclusion of 
 $G$ into $\nu(G)$ induces a homo\-mor\-phism $\imath: G \to \nu^q(G)$. We have 
 $g^{\imath \rho} = g$ and thus $\imath$ is injective. 
 Similarly the inclusion of $G^{\vfi}$ into $\nu(G)$ induces a monomorphism 
$\jmath: G^{\vfi} \to \nu^q(G)$. Thus we shall identify the elements $g \in G$ 
and $g^{\vfi} \in G^{\vfi}$ with their respective images $g^{\imath}$ and 
$(g^{\vfi})^{\jmath}$ in $\nu^{q}(G)$. 

Now let $\GG$ denote the subgroup of $\nu^q(G)$ generated by the 
images of $\widehat{\cG}$. By relations $(\ref{RR3})$,  $\GG$ 
normalizes the subgroup $T = [G, G^{\varphi}]$ in $\nu^q(G)$ and hence 
${\Upsilon}^{q}(G)= T\GG =[G, G^{\varphi}] \GG$ is a normal subgroup of 
${\nu}^q(G)$. Thus we obtain $\nu^q(G) = G^{\varphi} \cdot (G \cdot \Upsilon^q(G)),$ where the dots 
mean internal semidirect products. It should be noted that the actions of $G$ and 
$G^{\vfi}$ on $\Upsilon^q(G)$ are those induced by the defining relations of 
$\nu^q(G)$: for any elements $g, x \in G$, $h^{\vfi}, y^{\vfi} \in G^{\vfi}$ 
and $\widehat{k} \in \widehat{\mathcal{G}}$ we have 
$[g, h^{\vfi}]^{x} = [g^{x}, (h^{x})^{\vfi}]$ and 
$(\widehat{k})^{x} = \widehat{(k^x)}$; 
similarly, 
$[g, h^{\vfi}]^{y^{\vfi}} = [g^{y}, (h^{y})^{\vfi}]$ and 
$(\widehat{k})^{y^{\vfi}} = \widehat{(k^y)}.$ In addition, for any $\tau \in 
\Upsilon^q(G)$, $(g \tau)^{y^{\vfi}} = g [g, y^{\varphi}] \tau^{y^{\vfi}} \in G \Upsilon^q(G)$. 

By \cite[Proposition 2.9]{BR} ${\Upsilon}^{q}(G)$ is isomorphic to the $q$-tensor square 
$G \otimes^q G,$ for all $q \geq 0$.  
We then get a result (see \cite[Corollary 2.11]{BR}) 
analogous to one due to Ellis in \cite{Ellis}: 
$\nu^q(G) \cong G \ltimes (G \ltimes (G \otimes^q G));$ 
this generalizes a similar result found in \cite{Rocco1} for $q=0.$ 

The commutator approach to $G \otimes G$ for the case $q=0,$ provided by the isomorphism between 
$G \otimes G$ and the subgroup $[G, G^{\vfi}]$ of $\nu(G)$ 
(see \cite{Rocco1}, and also \cite{EL}), has proven suitable to treat of 
non-abelian tensor products of groups, Schur multipliers and many other relevant invariants 
involving covering questions in groups; see for instance, references 
\cite{EL}, \cite{Rocco2}, \cite{Nakaoka}, \cite{BM}, 
\cite{EN}, \cite{NR} and the GAP Package ``POLYCYCLIC'' in 
\cite{ENGAP}. The extension of the existing theory from $q=0$ to all non-negative integers $q,$ as 
addressed for instance in \cite{BR}, broadens the scope of these connections, now in a \textit{hat} 
(``power'') and \textit{commutator} approach to the $q$-tensor square, $q \geq 0$. 

In section 2 we extend to $G \otimes^q G, q \geq 0,$ some structural results found 
in \cite{BFM} and \cite{Rocco2} concerning $G \otimes G$. 
In section 3 it is  established an upper bound for the minimal number of generators of 
$G \otimes^q G$ when $G$ is a finitely generated nilpotent group of class 2, 
thus generalizing a result of Bacon found in \cite{Bacon}. 
We end by computing the $q$-tensor square 
of the free nilpotent group of rank $n \geq 2$ and class 2, 
$\cN_{n,2}$, $q \geq 0$; 
this will show, as in the case $q = 0$ (see \cite[Theorem 3.2]{Bacon}), 
that the cited upper bound is also attained for these groups when $q > 1$, 
although in this case $\cN_{n,2} \otimes^q \cN_{n,2}$ is a non-abelian group. 

Notation is fairly standard (see for instance \cite{Robinson}). 
If $x$ and $y$ are elements of a group $G$ 
then we write $y^x$ for the conjugate $x^{-1} y x$ 
and $[x, y]$ for the commutator $x^{-1} y^{-1} x y$. 
Our commutators are left normed: 
$[x, y, z] = [[x, y], z]$ for all $x,y,z \in G,$ 
and so on, recursively, for commutators of 
higher weights. 
The order of $x$ (resp. of $G$) is written 
$o(x)$ (resp. $|G|$). 
As usual, $\gamma_i(G)$ denotes the $\text{i}^{\underline{\text{th}}}$ term of 
the lower central series of $G$. 
For future reference we recall the well known 
Hall-Witt identity: 
\begin{equation} \label{HallWitt}
  [x, y^{-1}, z]^{y} [y, z^{-1}, x]^{z}  [z, x^{-1}, y]^{x} = 1, \; \forall x, y, z \in G.  
\end{equation}
In view of the isomorphism given by 
\cite[Proposition 2.9]{BR}, 
from now on we identify 
$G \otimes^q G$ with the subgroup 
$\Upsilon^q(G) = [G, G^\vfi] \GG \leq \nu^q(G)$ and write 
$[g, h^\vfi]$ in place of $g \otimes h$, for all 
$g, h \in G$. 
Following \cite{BR} we write 
$\Delta^q(G)$ for the subgroup 
$\langle [g, g^\vfi] \vert g \in G \rangle \leq \Upsilon^q(G),$ 
which by Lemma~\ref{lem:basic} (vii) is a central subgroup of $\nu^q(G).$ 
We write $\tau^q(G)$ for the factor group 
$\nu^q(G)/ \Delta^q(G)$. 
The subgroup $\Upsilon^q(G)/ \Delta^q(G)$ 
of $\tau^q(G)$ 
is isomorphic to the q-exterior square 
$G \wedge^q G$. 
In order to avoid any confusion we usually write 
$[G, G^\vfi]_{\tau^(G)}$ 
to identify the q-exterior square 
$G \wedge^q G$ with the image of 
$[G, G^\vfi]$ in $\tau^q(G)$. 
We shall eventually write $T$ to denote the subgroup $[G, G^\vfi]$ of ${\nu^q(G)}$ in order to 
distinguish it from the nonabelian tensor square $G \otimes G \cong [G, G^\vfi] \leq \nu(G)$ in the 
case $q =0$.

The material presented here incorporates part of the doctoral thesis \cite{tese} of the second named author, 
written under the supervision of the first.

\section{Some Structural Results}

In this section we extend results found in \cite{BFM} and \cite{Rocco2} related to the non-abelian 
tensor 
square, from $G \otimes G$ to $G \otimes^q G$, $q \geq 0$. We begin by including some previous, 
technical results for future references. 

The following basic properties are consequences of 
the defining relations of $\nu^{q}(G)$. 
\blem \cite[Lemma 2.4]{BR} 
\label{lem:basic}
Suppose that $q \geq 0$. The following relations hold in $\nu^{q}(G)$, for all $g, h, x, y \in G$.
\begin{itemize}
\item[\rm{(i)}] $[g, h^{\varphi}]^{[x, y^{\varphi}]} = [g, h^{\varphi}]^{[x, y]}$; 
\item[\rm{(ii)}] $[g, h^{\varphi}, x^{\varphi}] = [g, h, x^{\varphi}] = [g, 
h^{\varphi}, x] = [g^{\vfi}, h, x^{\vfi}] = [g^{\vfi}, h^{\vfi}, x] = [g^{\vfi}, h, x]$;
\item[\rm{(iii)}] If $h \in G'$ (or if $g \in G'$) then 
$[g, h^{\varphi}][h, g^{\varphi}]=1$;
\item[\rm{(iv)}] $[\widehat{x}, [g, h^{\varphi}]]= [\widehat{x}, [g, h]]$;
\item[\rm{(v)}] $(\widehat{x})^g = \widehat{x} [x^q, g^{\varphi}]$;
\item[\rm{(vi)}] If $[g, h]=1$ then $[g, h^{\varphi}]$ and $[h, g^{\varphi}]$ are central elements 
of $\nu^{q}(G)$, of the same finite order dividing $q$. If 
in addition $g, h$ are torsion elements of orders $o(g), o(h)$, respectively, then the order of 
$[g, h^{\varphi}]$ divides the $\gcd(q, o(g), o(h))$. 
\item[\rm{(vii)}] $[g, g^{\varphi}]$ is central in $\nu^{q}(G)$, for all $g \in G$;
\item[\rm{(viii)}] $[g, h^{\varphi}][h, g^{\varphi}]$ is central in $\nu^{q}(G)$;
\item[\rm{(ix)}] $[g, g^{\varphi}] = 1$, for all $g \in G^{\prime}$;
\item[\rm{(x)}] If $[x, g] = 1 = [x, h]$, then $[g, h, x^{\varphi}] = 1 = [[g, h]^{\varphi}, x]$.
\end{itemize}
\elem

\bcor 
\label{cor:basic2}
Let G be any group and let $g, h$ be arbitrary elements in $G$. Then
\begin{itemize}
   \item[(i)] $[G^\prime, G^\vfi] = [G, {G^\prime}^\vfi]$;
    \item[(ii)]  $[G^\prime, Z(G)^\vfi] = 1$;
    \item[(iii)] If $gG^\prime = hG^\prime$ then $[g, g^\vfi] = [h,h^\vfi]$;  
    \item[(iv)] If $o^\prime(x)$ denotes the order of a coset $x G^\prime  \in G/ G^\prime$, 
then  
$[g, h^\vfi][h, g^\vfi]$ has order dividing the 
$\gcd(q, o^\prime(g),o^\prime(h))$; 
    \item[\rm{(v)}]  The order of $[h, h^\vfi]$ divides the 
    $\gcd(q, o^\prime(h)^2, 2 o^\prime(h))$.
\end{itemize}
\ecor
\begin{proof}
   Part (i) follows directly from 
Lemma~\ref{lem:basic} (iii)   
 (see also \cite[Corollary 1.2 (iii)]{BFM}). As for part (ii), see  \cite[Proposition 2.7 (i)]{Rocco1}.The 
remaining parts 
are appropriate adaptations of \cite[Lemma 3.1 (v)]{Rocco2}, using 
\eqref{RR6} and Lemma~\ref{lem:basic} (vi). 
\end{proof}

For our purposes we establish the following proposition, which may have its own interest. 
\begin{prop} \label{prop:nclass}
  Let $G$ be a nilpotent group of class 2. Then the following hold in $\nu^q(G)$:
 \begin{itemize}
    \item[(i)] $\GG$ centralizes $[G, G^\vfi];$
    \item[(ii)] $[G^\prime, G^\vfi] \; 
\left(= [G,  {G^\prime}^\vfi] \right)$ is a central subgroup of $\nu^q(G)$;
    \item[(iii)] $\Upsilon^q(G) \; \left( \cong G \otimes^q G \,\right)$  is nilpotent of class at 
most 2.
 \end{itemize}
\end{prop}
\begin{proof}
   (i) follows straightforward from \ref{lem:basic} (iv) and relation \eqref{RR1}, once $G$ has nilpotency 
class 2.  

\noindent
(ii). For all $g, h \in G$ and $c \in G^\prime$ we have: 
\begin{align*}
    [c, g^\vfi]^h &= [c, g^\vfi][c, g^\vfi, h]  &\  \\
    \ &= [c, g^\vfi][c, g, h^\vfi] &\quad 
     \text{(by Lemma~\ref{lem:basic}, (i))} \\
    \ &= [c, g^\vfi]  &\quad \text{(since} \; [c, g] = 1, \; \text{as} \; G^\prime \leq Z(G)) \\
    \ &=[c, g^\vfi]^{h^\vfi}  &\quad (\text{by definition of} \; \nu^q(G)).
\end{align*} 
In addition, for all $\widehat{k} \in \widehat{\cG},$ by Lemma~\ref{lem:basic} (iv) and relations 
\eqref{RR1} we have that  
$\widehat{k}^{[c, g^\vfi]} =  \widehat{k}^{[c, g]} = \widehat{(k^{[c, g]})} = \widehat{k},$ since 
$[c, g] = 1$.
 This proves part (ii) (using the definition of $\nu^q(G)$), because $[G^\prime, G^\vfi]$ 
 is generated by all those $[c, g^\vfi]$ above.  

\noindent
(iii). That $\Upsilon^q(G)$ is nilpotent and has nilpotency class at most 3 follows from 
\cite[Proposition 2.7, (i)]{BR}. 
Now, $\Upsilon^q(G) = [G, G^\vfi]\GG$ and thus, once $\GG$ centralizes $[G, G^\vfi],$ we have  
   $$(\Upsilon^q(G))^\prime = [[G, G^\vfi]\GG, [G, G^\vfi]\GG] \
   \leq [G, G^\vfi]^\prime [\GG, \GG].$$
Induction arguments can be used, together with Lemma~\ref{lem:basic} (i), (ii), (iv) and (v) and 
defining relations \eqref{RR5} -- \eqref{RR6} to get: 
\begin{itemize}
   \item[(a)] $[G, G^\vfi]^\prime = [G^\prime, (G^\prime)^\vfi]$ \ 
(see also \cite[Proposition 1.3 (i)]{BFM} or \cite[Theorem 3.3]{Rocco1});
    \item[(b)] $[\GG, \GG] \leq 
\langle \widehat{\cG^\prime} \rangle [G^\prime, G^\vfi] \leq [G^\prime, G^\vfi].$ 
\end{itemize}
Consequently, $(\Upsilon^q(G))^\prime \leq [G^\prime, G^\vfi],$
which by part (ii) is central in $\nu^q(G)$. This completes the proof.  
\end{proof}

For a finitely generated abelian group $A,$ its q-tensor square $\Upsilon^q(A)$ can be computed by 
repeated applications of the  following two results from \cite{BR}.

\begin{lem} \cite[Corollary 2.16]{BR} \label{lem:directprod}
Let $G = N \times H$ be a direct product and set 
$\overline{N} = N/{N'N^q}, \; \overline{H} = H/{H'H^q}$. 
Then 
\begin{itemize}
\item[$(i)$] $\Upsilon^q(G) = \Upsilon^q(N) \times [N, H^{\varphi}] [H, N^{\varphi}] \times 
\Upsilon^q(H);$
\item[$(ii)$] $[N, H^{\varphi}] \cong (\overline{N} \otimes_{\Z_q} \overline{H}) \cong 
[H, N^{\varphi}]$. 
\end{itemize}
\end{lem}

\begin{lem} \cite[Theorem 3.1]{BR} \label{lem:cyclic}
Let $C_n$ (resp. $C_{\infty}$) be the cyclic group of order $n$ (resp. $\infty$), $q$ a non-
negative integer and $d = \gcd(n, q).$ Then 
\begin{align*}
C_{\infty} \otimes^q C_{\infty} & \cong \; \; C_{\infty} \times C_q;  \\
C_n \otimes^q C_n \; \; & \cong 
\begin{cases} 
C_n \times C_d, & \text{if \; $d$ is odd}; \\
C_n \times C_d, & \text{if \; $d$ is even and either $4|n$ or $4|q$}; \\
C_{2n} \times C_{d/2}, & \text{otherwise.} 
\end{cases}
\end{align*}
\end{lem} 

Thus, if $A = \prod^r_{i=1} C_i$ is a direct product of the cyclic groups $C_i, \, i = 1, \ldots, 
r$, where $C_i = \langle x_i \rangle$, then 
\[
   \Upsilon^q(A) = \prod_{i=1}^r \Upsilon^q(C_i) \times 
\prod_{1 \leq i < j \leq r}[C_i, C^\vfi_j][C_j, C^\vfi_i].
\]
Here we have $\Upsilon^q(C_i) = \langle [x_i, x^\vfi_i], \; \widehat{x_i} \rangle$ and 
$[C_i, C^\vfi_j][C_j, C^\vfi_i] = \langle \, 
[x_i, x^\vfi_j][x_j, x^\vfi_i], \; [x_i, x^\vfi_j] \, \rangle$. 
Since $\Delta^q(A) = \langle \, [a, a^\vfi] \, \vert \, a \in A \rangle$,  
we observe, like in \cite[Proposition 3.3]{Rocco2},
 that $\Delta^q(A) = \langle [x_i, x^\vfi_i], \, 
[x_j, x^\vfi_k][x_k, x^\vfi_j] \, \vert \, 1 \leq i \leq r, \, 1 \leq j < k \leq r \rangle$  
and thus it does not depend on the particular set $X = \{x_1, \ldots, x_r \}$ of generators of $A$. 
Consequently, we can write 
\[
  \Upsilon^q(A) = \Delta^q(A) E^q_{X}(A),                                                           
 \]
where $E^q_{X}(A) = \langle \widehat{x_i}, \, [x_j, x^\vfi_k] \, \vert \, 
1 \leq i \leq r, \, 1 \leq j < k \leq r \rangle.$  

\begin{rem}
   If $x$ and $y$ are commuting elements in any group 
$G$ then by relations \eqref{RR4}--\eqref{RR6} 
and Lemma~\ref{lem:basic} (vi) we get 
\[
\widehat{x y} = 
\widehat{x} \, \widehat{y} \, [x, y^\vfi]^{-{q \choose{2}}} = 
\widehat{x} \, \widehat{y} \, [y, x^\vfi]^{-{q\choose{2}}} = \widehat{y x},                          
\]
and hence $[x, y^\vfi]^{-{q\choose{2}}} = 
[y, x^\vfi]^{-{q\choose{2}}}.$ In particular, if $q = 2$ then $[x, y^\vfi] = [y, x^\vfi].$ This 
means for instance that in the  decomposition of $\Upsilon^q(A)$ found above, the groups 
$[C_i,  C^\vfi_j]$ and $[C_j, C^\vfi_i]$ are not necessarily independent.   
Moreover, the identity 
$\widehat{(x^n)} = (\widehat{x})^n [x, x^\vfi]^{-{n \choose 2}{q \choose 2}}$ shows that the 
subgroups 
$\langle \widehat{x_i} \rangle$ and $\langle [x_i, x^\vfi_i] \rangle$ 
of 
$\Upsilon^q(C_i)$ 
may have non trivial intersection. Consequently, unlike the case $q = 0$, the 
subgroup 
$E^q_{X}(A)$ is not necessarily a complement of
$\Delta^q(A)$ (see also  \cite[Section 2]{BFM}).
\end{rem}

Now let $G$ be any group and write $G^{ab} = G/ G^\prime$. The natural projection 
$G \twoheadrightarrow G^{ab}$ induces an epimorphism  
$\pi : \nu^q(G) \to \nu^q(G^{ab})$. We denote by $\pi_{0}$ the restriction of $\pi$ to 
$\Upsilon^q(G).$ By \cite[Lemma 2.14 (iii)]{BR} we have that 
$\Ker(\pi_{0}) = [G^{\prime}, G^{\vfi}][G, {G^\prime}^{\vfi}] \langle \widehat{\mathcal{G}^{\prime}}
 \rangle,$ which reduces to 
$[G^{\prime}, G^\vfi] \langle \widehat{\mathcal{G}^{\prime}} \rangle,$ by force of 
Corollary~\ref{cor:basic2} (i). In 
addition, using relations \eqref{RR4} and \eqref{RR6}, an induction argument as in 
the proof of the Proposition\ref{prop:nclass} (ii)  shows 
that $\widehat{(\mathcal{G}^{\prime})} \leq [G^{\prime}, G^\vfi]$ and, consequently,
$\Ker(\pi_0) = [G^\prime, G^\vfi].$ 
 
The next Lemma extends \cite[Lemma 2.1]{BFM} to all $q \geq 0$ 
(see also \cite[Proposition 3.3]{Rocco2}). We shall omit the proof. 
\begin{lem} \label{lem:fg1}
   Let $q$ be a non negative integer and $G$ be a group such that $G^{ab}$ is finitely generated. 
Assume that $G^{ab}$ is a direct product of the cyclic groups $C_i = \langle x_i G^\prime \rangle,$ 
for $i = 1, \ldots, r$ and set 
$$E^q(G) = \langle \widehat{x_i}, \, [x_j, x^\vfi_k] \, \vert \, 1 \leq i \leq r, \, 1 \leq j < k \leq r \
rangle \, 
[G^\prime, G^\vfi].$$ Then,  
\begin{itemize}
   \item[(i)] $\Delta^q(G) = \langle [x_i, x^\vfi_i], \, [x_j, x^\vfi_k][x_k, x^\vfi_j] \, \vert \, 
1 \leq i \leq r, \, 1 \leq j < k \leq r \rangle;$
    \item[(ii)] $\Upsilon^q(G) = \Delta^q(G) E^q(G).$
\end{itemize}
\end{lem}

With the above notation, let $\pi_1$ denote the restriction of $\pi_0$ to $\Delta^q(G)$,  
$\pi_1 : \Delta^q(G) \twoheadrightarrow \Delta^q(G^{ab}),$ and let $N = \Ker(\pi_1)$. Therefore, $N 
= \Delta^q(G) \cap  [G^\prime, G^\vfi]$  $\left(= \Delta^q(G) \cap  E^q(G) \right),$ a central 
subgroup of $\Upsilon^q(G)$.  

  Our next theorem generalizes, to all $q \geq 0,$ Proposition 2.2 in \cite{BFM}, which in turn 
improves Proposition 3.3 in \cite{Rocco2}. 
\begin{thm} \label{thm:fg2}
   Let $q \geq 0$ and assume that $G^{ab}$ is finitely generated. Then, with the 
notation of Lemma~\ref{lem:fg1}, the following hold: 
  \begin{itemize}
     \item[(i)] $\Upsilon^q(G)/ N \cong \Delta^q(G^{ab}) \times (G \wedge^q G)$;
    \item[(ii)] If $q \geq 1$ and $q$ is odd, then $N = 1$ and thus 
 $\Delta^q(G) \cong \Delta^q(G^{ab})$ and $\Upsilon^q(G) \cong \Delta^q(G^{ab}) \times 
(G \wedge^q G);$ 
    \item[(iii)] For $q = 0$ or $q \geq 2$ and $q$ even, if $G^{ab}$ has no element of order two or 
if $G^\prime$ has a complement in $G$, then also $N = 1,$ $\Delta^q(G) \cong \Delta^q(G^{ab})$ and 
$\Upsilon^q(G) \cong \Delta^q(G^{ab}) \times (G \wedge^q G);$
    \item[(iv)] For $q \geq 2$ and $q$ even, if $G^{ab}$ has no element of order 2, then 
$\Delta^q(G)$ is a homocyclic abelian group of exponent $q,$ of rank $t + 1 \choose 2$;
    \item[(v)] If $G^{ab}$ is free abelian of rank $t$, then the conclusion of the previous item 
holds for all $q \geq 1,$ while $\Delta^q(G)$ is free abelian of rank  
$t + 1 \choose 2$ if $q = 0$. 
  \end{itemize}
\end{thm}
 \begin{proof}
    (i): By Lemma~\ref{lem:fg1} (ii) we have 
\[
  \frac{\Upsilon^q(G)}{N} = \frac{\Delta^q(G) E^q(G)}{N}  \cong \frac{\Delta^q(G)}{N} \times 
\frac{E^q(G)}{N}.
\]
Now, $\Delta^q(G)/ N \cong \Delta^q(G^{ab})$ and 
\[
E^q(G)/ N = E^q(G)/ (\Delta^q(G) \cap E^q(G)) \cong  \Upsilon^q(G)/ \Delta^q(G) 
\cong G \wedge^q G.
\]
This proves (i).

\vskip 10pt
\noindent
(ii), (iii), (iv) and (v): Suppose that the torsion subgroup of $G^{ab}$ is the direct product of 
the 
cyclic groups $\langle x_i G^\prime \rangle$ of order $n_i, \, 1 \leq i \leq s,$ and let the free 
part of $G^{ab}$ be the direct product of the cyclic groups 
$\langle y_j G^\prime \rangle, \, 1 \leq j \leq t$. 
Thus, $o^\prime(x_i) = n_i$ and $o^\prime(y_j) = \infty$. 
Set $X : = \{ x_i \, \vert \, 1 \leq i \leq s \}$ and $Y : = \{ y_j \, \vert \, 
1 \leq j \leq t \}.$ Then $G$ is generated by 
$X \cup Y \cup G^\prime$. Using Lemma~\ref{lem:basic} and 
Corollary~\ref{cor:basic2} (see also \cite[Proposition 3.3 and Remark 5]{Rocco2}) 
we find that $\Delta^q(G)$ is generated by the set 
$\Delta_{X} \cup \Delta_{Y} \cup \Delta_{XY}$, where
\begin{align*}
   \Delta_X & = \{ [x_i, x^\vfi_i], \, [x_j, x^\vfi_k][x_k, x^\vfi_j] \, \vert \, 
1 \leq i \leq s, \, 1 \leq j < k \leq s \}, \\
  \Delta_Y & = \{ [y_j, y^\vfi_j], \, [y_k, y^\vfi_l][y_l, y^\vfi_k] \, \vert \, 
1 \leq j \leq t, \, 1 \leq k < l \leq t \}, \\
 \Delta_{XY} & = \{[x_i, y^\vfi_j][y_j, x^\vfi_i] \, \vert \, 1 \leq i \leq s, 
\, 1 \leq j \leq t \}. 
\end{align*}
Set $n_{ik} = \gcd(n_i, n_k).$ Parts (iv) and (v) of Corollary~\ref{cor:basic2} 
give 
$([x_i, x^\vfi_k][x_k, x^\vfi_i])^{n_{ik}} = 1$ and  $([x_i, y^\vfi_j][y_j, 
x^\vfi_i])^{n_i} = 1$, while 
$[x_i, x^\vfi_i]^{n_i} \in \Ker(\pi_0)$, $\forall i, k = 1, \ldots, s, i < k, 
\forall j = 1, \ldots t$. 
Actually, $\Ker(\pi_0)$ is generated by the set $\{[x_i, x^\vfi_i]^{n_i} \, \vert \, 1 \leq i \leq 
s \},$ by \cite[Proposition 3.5]{Rocco2}. Again by Corollary~\ref{cor:basic2} (v), we get that if 
$n_i (= o^\prime(x_i))$ is odd, then $[x_i, x^\vfi_i]^{n_i} = 1$, while  $[x_i, x^\vfi_i]^{2 n_i} = 
1$ if $n_i$ is even. Consequently, $N = \Ker(\pi_0)$ is an elementary abelian 2-group of rank at 
most $r_2(G^{ab})$, the $2-rank$ of $G^{ab}$ (see also \cite[Corollary 3.6]{Rocco2}). On the other 
hand, we should take into account that $q$ is involved in the upper bound found in Corollary~\ref{cor:basic2} (
v). 
Thus, if $q \geq 1$ and $q$ is odd, then 
$\gcd(q, 2n_i) = \gcd(q, n_i) \mid n_i$ and hence $[x_i,x^\vfi_i]^{n_i} = 1,$ 
for all $i = 1,\ldots, s.$ Therefore $N = 1$ in this case, proving part (ii). 
It should be also clear that $N = 1$ if $r_2(G^{ab}) = 0.$ This proves (iii) in 
the case where $G^{ab}$ has no element of order 2. Now, if $G^\prime$ has a 
complement $C$ in $G,$ then every $g \in G$ can be written as $g = xh$ with 
$x \in C$ and $h \in G^\prime$. 
Corollary~\ref{cor:basic2} (iii) says that 
$[g, g^\vfi] = [x, x^\vfi]$ and thus 
$\Delta^q(G) = \langle [x, x^\vfi] \, \vert \, x \in C \rangle = 
\Delta^q(G^{ab}).$ 
This completes the prof of part (iii) (see also \cite[Proposition 2.2]{BFM}). 
Finally, we observe that 
$[x_i, x^\vfi_i] = 1 = [x_i, y^\vfi_j][y_j, x^\vfi_i]$ 
in the case where $r_2(G^{ab}) = 0$ and $q \geq 2$, $q$ even. 
Here we have $\Delta^q(G) = \langle \Delta_Y \rangle$ 
and $[y_j, y^\vfi_j]^q = 1 = ([y_k, y^\vfi_l][y_l, y^\vfi_k])^q, 
\forall j, k, l = 1, \ldots, t, \, k < l.$ 
That each of these $t+1 \choose 2$ generators has order $q$ follows immediately 
from Lemma~\ref{lem:cyclic}, where we found $C_{\infty} \otimes^q C_{\infty} \cong C_{\infty} \times C_q$. 
Part (v) follows by an analogous argument , as in 
part (iv); the last assertion can be also found in \cite[Corollary 3.6]{Rocco2}. 
The proof is complete.   
\end{proof}

We state the next Lemma for easy of reference, which in a certain sense extends ideas found in 
\cite{Miller} for the case $q = 0.$ 
A proof is given in \cite{BL} for $q=0$ (see also  \cite[Proposition 3.2]{BFM} for an 
alternative proof for this case) 
and in \cite{ER89} for $q \geq 1$.  
\begin{lem} \label{lem:qext}
	Let $F/R$ be a free presentation of a group $G$. Then 
	\[
	G \wedge^q G \cong F^\prime F^q / [R, F] R^q.
	\]
\end{lem}

Notice that there is a map 
\[
\rho: \nu^{q}(G) \lra G, \, g  \longmapsto g, \, g^{\varphi} \longmapsto g \; 
\text{and} \; \widehat{k} \longmapsto  k^q.
\] 
Let $\rho^{\prime} = \rho{|}_{\Upsilon^{q}(G)} : 
\Upsilon^{q}(G) \lra G, 
[g_1,{g_2}^{\varphi}] \longmapsto  [g_1,g_2], \widehat{k} \longmapsto  k^q.$ 

\noindent
Following \cite{BR} we 
write $\theta^q(G) = \Ker(\rho)$ and $\mu^q(G) = \Ker(\rho^{\prime}) = \Upsilon^q(G) \cap \theta^q(G).$ 
It follows that $\Upsilon^q(G)/ \mu^q(G) \cong G^{\prime} G^q.$ If $G = F /R$ is a free 
presentation of $G$, then 
\begin{equation} \label{eq:Schur}
H^2(G, \Z_q) \cong R \cap F^\prime F^q / R^q [R, F] = 
(G \wedge^q G) \cap M^q(G),
\end{equation}
where $M^q(G)= R / R^q [R, F]$ is the \textit{q-multiplier} of $G$. From this we obtain 
(see for instance \cite[Theorem 2.12]{BR}): 
\begin{equation} \label{eq:h2G}
\mu^q(G) / \Delta^q(G) \cong H^2(G, \Z_q),    
\end{equation}
for all $q \geq 0$.

\begin{cor} \label{cor:free}
   Let $F_n$ be the free group of rank $n$. Then 
\begin{itemize}
   \item[(i)] For $q \geq 1,$
   \[
  F_n \otimes^q F_n \cong C^{n+1 \choose 2}_q \times (F_n)^\prime (F_n)^q.
  \]
  \item[(ii)] (\cite[Proposition 6]{BJR}) For $q = 0,$ 
  \[
  F_n \otimes F_n \cong C^{n+1 \choose 2}_{\infty} \times (F_n)^{\prime}.
  \]
\end{itemize}
\end{cor}
\begin{proof}
 Since $F^{ab}_n$ is free abelian of rank $n$, by Theorem~\ref{thm:fg2} (ii), (iii) and (v), we 
have:
 \[
\Upsilon^q(F_n) \cong \Delta^q(F^{ab}_n) \times (F_n \wedge^q F_n).
\]
 (i): If $q \geq 1$ then 
$\Delta^q(F^{ab}_n) \cong C^{n+1 \choose 2}_q$ and,  
by Lemma~\ref{lem:qext} (i) with  $G = F_n$ and $R = 1,$  
$$F_n \wedge^q F_n \cong (F_n)^\prime (F_n)^q.$$ 
This proves (i). 

\noindent
(ii): If $q = 0$ then 
$\Delta^q(F^{ab}_n) \cong C^{n+1 \choose 2}_{\infty}$ and, again by the previous Lemma,  
$$F_n \wedge F_n \cong (F_n)^\prime.$$ 
This completes the  proof.   
\end{proof}
\begin{cor} \label{cor:freenil}
   Let $\mathcal{N}_{n,c} = F_n / \gamma_{c+1}(F_n)$ be the free nilpotent group of class $c \geq 
1$ and rank $n > 1$. Then
\begin{itemize}
   \item[(i)] For $q \geq 1,$
   \[
  \mathcal{N}_{n,c} \otimes^q \mathcal{N}_{n,c} \cong 
C^{n+1 \choose 2}_q \times 
\frac{(F_n)^\prime (F_n)^q}{\gamma_{c+1}(F_n)^q 
\gamma_{c+2}(F_n)}.
  \]
  \item[(ii)] (\cite[Corollary 1.7]{BMM}) For $q = 0,$ 
  \[
  \mathcal{N}_{n,c} \otimes \mathcal{N}_{n,c} \cong C^{n+1 \choose 2}_{\infty} \times 
(\mathcal{N}_{n,c+1})^{\prime}.
  \]
\end{itemize}
\end{cor}
\begin{proof}
   (i) and (ii) follow by similar arguments as above, taking into account that here we have 
$R = \gamma_{c+1}(F_n)$ 
and thus $[R, F]$, as in Lemma~\ref{lem:qext}, is precisely $\gamma_{c+2}(F_n)$.
\end{proof}

\section{q-Tensor Squares of Nilpotent Groups of Class 2}

In this section we restrict our considerations to finitely generated nilpotent groups $G$ of class two. We 
begin with a 
general result concerning polycyclic groups found in 
\cite{BR}; this generalizes to all $q \geq 0$ a result due to Blyth and Morse in \cite{BM} for $q=0$, which in 
turn extends 
to all polycyclic groups a similar result for finite solvable groups found in \cite{Rocco2}.   

\begin{lem} (\cite[Corolary 3.6]{BR}) \label{lem:poly} 
Let $G$ be a polycyclic group with a polycyclic generating sequence $\pgs(G) = \left( {\fa}_1, 
\ldots,  {\fa}_n \right).$  Then 
\begin{itemize}
\item[(i)] $[G, G^\vfi]$, a subgroup of $\nu^q(G), \, q \geq 0$, is generated by 
\[
[G, G^\vfi] = \biggl\langle [\fa_i, \fb_i], [\fa_i, \fb_j][\fb_j, \fa_i], 
[\fa^{\alpha}_i, (\fb_j)^{\beta}],   
\text{for $1 \leq i < j \leq n$}, \, 1 \leq k \leq n \biggr\rangle, 
\]
\item[(ii)] $\Upsilon^q(G)$, a subgroup of $\nu^q(G), \, q \geq 1$, is generated by 
\[
\Upsilon^q(G) = \biggl\langle [\fa_i, \fb_i], [\fa_i, \fb_j][\fb_j, \fa_i], 
[\fa^{\alpha}_i, (\fb_j)^{\beta}], \widehat{(\fa_k)}, \; 
\text{for $1 \leq i < j \leq n$}, \, 1 \leq k \leq n  \biggr\rangle, 
\]
\footnotesize{where    
 $\alpha = 
\begin{cases} 
1 & \text{if} \ o(\fa_i) < \infty \\
\pm 1 & \text{if} \ o(\fa_i) = \infty
\end{cases} \; \text{and} \;
\beta = 
\begin{cases} 
1 & \text{if} \ o(\fa_j) < \infty \\
\pm 1 & \text{if} \ o(\fa_j) = \infty.
\end{cases}$}
\vskip 10pt
\item[(iii)] $\Delta^q(G)$ is generated by the set 
$\{ [\fa_i, \fb_i], [\fa_i, \fb_j][\fb_j, \fa_i], \; \text{for} \; 1 \leq i < j \leq n \}.$
\end{itemize} 
\end{lem} 

Now let $G$ be a finitely generated nilpotent group of class two and assume that $G$ is generated 
by $g_1, g_2, \cdots, g_n.$ 
Thus, any element $g \in G$ can be written as 
\begin{equation} \label{eq:expression}
   g = \prod_{i = 1}^{n} g^{m_i}_{i} \prod_{1 \leq j < k \leq n} [g_j, \, g_k]^{l_{jk}},
\end{equation} 
where the exponents $m_i$ and $l_{jk}$ are integers. Consequently, $G$ has the following polycyclic 
generating set 
\begin{equation} \label{eq:polycyclic1}
   \{ g_i, \, 1 \leq i \leq n \} \cup \{ [g_j, \, g_k], \, 1 \leq j < k \leq n \}. 
\end{equation}

The following theorem extends a result of Bacon in 
\cite[Theorem 3.1]{Bacon} (see also \cite{BK}) for all 
$q \geq 0$. We provide a proof for the case $q = 0$ using the commutator approach; the general case follows 
straightforward 
from this case and Lemma~\ref{lem:poly} (ii), but we shall prove it in this case too, for the sake of 
completeness.    

\begin{thm} \label{thm:Bacon} 
Let $G$ be a nilpotent group of class two with $d(G) = n$, then 
\begin{itemize}
   \item[(i)] (\cite[Theorem 3.1]{Bacon}) $d([G,G^{\vfi}]) \leq \frac{n(n^2 + 3n -1)}{3}$;
    \item[(ii)] $d(G \otimes^q G) \leq \frac{n(n^2+3n+2)}{3}$, for all $q \geq 0$; 
    \item[(iii)] In particular, if $G$ has finite exponent $e(G)$ and $\gcd(q, e(G)) = 1$ , 
then \\ $d(G \otimes^q G) \leq n^2$.
\end{itemize}
\end{thm}
\begin{proof}
 On assuming that $G$ is generated by $g_1, \ldots, g_n$ then we obtain the polycyclic generating 
set given 
by \eqref{eq:polycyclic1}. Thus, by Lemma~\ref{lem:poly} (i) we have that $\Upsilon(G) = [G, G^{\vfi}]$ is 
generated by the 
following set of elements: 
\begin{gather*}
\{[g^{\alpha}_i, (g^{\vfi}_j)^{\beta}] \, : \, 1 \leq i, j \leq n \} \cup \\  
\{ [g^{\alpha}_i, ([g_j, g_k]{\vfi})^{\beta}] \, : \, 1 \leq i \leq n, \; 1 \leq j < k \leq n \}  \cup  \\ 
\{[[g_j , g_k]^{\beta}, (g^{\vfi}_i)^{\alpha}] \, : \, 1 \leq i \leq n, \; 1 
\leq j < k \leq n \} \cup  \\ 
\{ [[g_r, g_s]^{\alpha}, ([g_t, g_u]^{\vfi})^{\beta}] \, : \, 1 \leq r < s \leq 
n, \; 1 \leq t < u \leq n \},  \quad \text{where} \quad  \alpha, \beta \in \{-1, 1 \}.
\end{gather*} 
Now by Lemma~\ref{lem:basic}, parts (ii), (iii), (ix), (x), and the fact that $G$ has 
class 2, we can further reduce the above set to obtain  
\begin{equation} \label{generators}
   \{[g_i, g^{\vfi}_j] \, : \, 1 \leq i, j \leq n \} \cup \{[g_i, [g_j, g_k]^\vfi] \, : \, 1 \leq i 
\leq n, \; 1 \leq j < k \leq n \}.
\end{equation}
This new set has $n^2$ generators of the form 
$[g_i, g^{\vfi}_j]$ and 
$n(n-1)$ generators of the form $[g_i, [g_j, g_i]^{\vfi}].$ It remains to count 
the generators of the form $[g_i, [g_j, g_k]^{\vfi}],$ when $i, j, k$ are all 
distincts and $j < k$. 
Now by \cite[Corollary 3.2]{Rocco1} $\nu(G)$ has nilpotency class at most 3 and thus, again 
\[
[g^{\vfi}_j, g^{\vfi}_k, g_i] [g^{\vfi}_k, g^{\vfi}_i, g_j] [g_i, g_j, g^{\vfi}_k] = 1. 
\] 
It then follows that $[g_i, [g_j, g_k]^{\vfi}] = [g_j, [g_i, g_k]^{\vfi}] 
[g_k, [g_j, g_i]^{\vfi}]$. Therefore, 
\[
d([G, G^{\vfi}]) \leq \frac{1}{3} n(n^2 + 3n -1).
\] 
This proves part (i). 

\noindent
(ii): Part (i) also proves (ii) in case $q = 0,$ giving us the better bound for $d(G \otimes G)$. Thus, we 
shall assume 
$q \geq 1.$ Since $\Upsilon^q(G) = [G, G^{\vfi}] \GG$ it suffices to control the number of generators of both 
$[G, G^{\vfi}]$ and $\GG$. We already know by part (i) 
that $[G, G^{\vfi}]$ is generated by the 
set $\{ [g_i, g^{\vfi}_i] \, : \, 1 \leq i,j \leq n\} 
\cup \{[g_i, [g_j, g_k]^{\vfi}] \, : \, 1 \leq i 
\leq n, \; 1 \leq j < k \leq n \}.$ 
Now, by definition the subgroup $\GG$ is 
generated by 
$\widehat{\mathcal{G}} = \{\widehat{g}, g \in G\}$. 
By the defining relations~\eqref{RR4} of $\nu^q(G)$ we have $\widehat{gh} = 
\wg 
\, ( \prod^{q-1}_{i=1}[ g, ( h^{\vfi} )^{-i} ]^{g^{q-1-i}} ) \, \wh,$ for all 
$g, h 
\in G,$ and hence $\widehat{gh} \equiv \wg \, \wh \, 
\pmod [G, G^\vfi],$ for all 
 $\wg, \wh \in \widehat{\mathcal{G}}.$ 
Since every element $g \in G$ has a unique expression in the form~\eqref{eq:expression} and given the 
fact that, for every commutator $[g_j, g_k] \in G',$ 
$\widehat{ [ g_j, g_k ] } = [g_j, g^\vfi_k]^q \in [ G, G^\vfi ]$ (by relations~\eqref{RR6}), we see in later 
stage that $\GG$ 
is 
generated, modulo $[G, G^{\vfi}]$, by the $n$ elements $\widehat{g_1}, \ldots, 
\widehat{g_n}.$ Therefore, we 
conclude  that 
$d(\Upsilon^q(G)) \leq \frac{1}{3}(n^3 + 3n^2 + 2n).$ 
This proves part (ii). 

\noindent
(iii): If in particular $G$ has finite exponent $e(G)$ and $\gcd(q, e(G)) = 1$, then we see by 
Lemma~\ref{lem:basic} (vi) 
that all generators of the forms $[g_i, g^{\vfi}_i]$ 
and $[g_i, [g_j, g_k]^{\vfi}]$ are trivial. Consequently, 
$d(G \otimes^q G) \leq n^2$ 
in this case. The proof is complete. 
\end{proof}
 
In \cite{Aboughazi} Aboughazi computed the nonabelian tensor square of the Heisenberg group 
$\mathcal{H} =  F_2/ \gamma_3(F_2)$, where $F_2$ denotes the free group of rank 2. 
There, it is found that $\cH \otimes \cH \cong \mathbb{Z}^6$, 
thus showing that the bound in Theorem~\ref{thm:Bacon} is sharp. 
Later, Bacon in \cite[Theorem 3.2]{Bacon} computed  
$\mathcal{N}_{n,2} \otimes \mathcal{N}_{n,2}$ for all $n \geq 2$, to show that the bound is also reached for 
the free $n-$generator nilpotent group of class 2, $n > 1$:  
$\mathcal{N}_{n,2} \otimes \mathcal{N}_{n,2}$ is a free abelian group of 
rank $\frac{1}{3} n (n^2 + 3n -1)$. 

It is not difficult to extend these results to the $q-$tensor square $\Upsilon^q(\mathcal{N}_{n,2}), q \geq 1,$ 
to show that the bound in Theorem~\ref{thm:Bacon} (ii) is also attained. 
In fact, the next proposition is but a specialization 
of Corollary~\ref{cor:freenil}.  We write $M(G)$ to denote the Schur multiplier $H_2(G, \mathbb{Z})$ of $G$. 
\begin{prop}
   Let $\cN_{n,2}$ be the free nilpotent group of rank $n > 1$ and class 2, $\cN_{n,2} = F_n / \gamma_3(F_n).$ 
Then, 
\begin{itemize}
  \item[(i)] (\cite[Theorem 3.2]{Bacon}) $\cN_{n,2} \otimes \cN_{n,2}$ is free abelian of rank 
$\frac{1}{3} n (n^2 + 3n -1).$ 
More precisely, 
 $$\cN_{n,2} \otimes \cN_{n,2} \cong \Delta(F^{ab}_n) \times M(\cN_{n, 2}) \times \cN^\prime_{n,2}.$$
    \item[(ii)] $\cN_{n,2} \otimes^q \cN_{n,2} \cong 
(C_q)^{({n + 1 \choose 2} + M_n(3))} \times 
\cN^\prime_{n,2} \cN^q_{n,2},$ 
where $M_n(3) = \frac{1}{3}(n^3 - n)$ is the $q-$rank of  $\gamma_3(\cN_{n,2})/ \gamma_3(\cN_{n,2})^q 
\gamma_4(\cN_{n,2}),$  according to the Witt's formula. Consequently, for $q > 1$  
$$d(\cN_{n,2} \otimes^q \cN_{n,2}) = 
\frac{1}{3}(n^3 + 3n^2 + 2n).$$
\end{itemize}
\end{prop}
\begin{proof}
   As in the proof of Corollary~\ref{cor:freenil}, using Lemma~\ref{lem:qext} we can write 
\[
G \wedge^q G \cong F^\prime_n F^q_n / [R, F_n] R^q.
\]
In view of \eqref{eq:Schur},
$$H^2(G, \Z_q) \cong R \cap F^\prime_n F^q_n / R^q [R, F_n].$$ 

\noindent
Taking into account that $R = \gamma_3(F_n) \leq F^\prime_n 
F^q_n$ we have the  exact sequence 
\begin{equation} \label{eq:Schur2}
1 \to \frac{\gamma_3(F_n)}{\gamma_4(F_n) \gamma_3(F_n)^q} 
\to \frac{F^\prime_n F^q_n}{\gamma_4(F_n) \gamma_3(F_n)^q} 
\to \frac{F^\prime_n F^q_n}{\gamma_3(F_n)} \to 1.
\end{equation}
Here we find that $H_2(\cN_{n,2}, \mathbb{Z}_q) \cong 
\gamma_3(F_n)/ \gamma_4(F_n) \gamma_3(F_n)^q \cong {\mathbb{Z}}^{M_n(3)}_q$, 
for all $q \geq 0,$ where, by the Witt's formula for the rank of 
$\gamma_r(F_n)/ \gamma_{r+1}(F_n),$ 
$$ 
M_n(r) = \frac{1}{r} \sum_{d \mid r} \mu(d) n^{\frac{r}{d}},
$$ 
where $\mu$ denotes the M\"obius function. 

\noindent
(i): If $q = 0$ then the exact sequence \eqref{eq:Schur2} splits and thus we have, by 
Corollary~\ref{cor:freenil}, 
\[
  \mathcal{N}_{n,c} \otimes \mathcal{N}_{n,c} \cong 
C^{n+1 \choose 2}_{\infty} \times \gamma_3(F_n) / \gamma_4(F_n) \times 
\gamma_2(F_n) / \gamma_3(F_n)  \cong C^{n+1 \choose 2}_{\infty} 
\times C^{n \choose 2}_{\infty} \times C^{\frac{1}{3}(n^3 - n)}_{\infty}.
\]
This is the result of Bacon, that $d(\cN_{n,2} \otimes \cN_{n,2}) = \frac{1}{3}n(n^2 + n -1)$. 

\noindent
(ii): If $q > 1$ then we see by the generators of $\cN_{n,2} \otimes^q \cN_{n,2}$ found in 
Theorem~\ref{thm:Bacon} that the 
image by $\rho^\prime$ of the subgroup 
$\langle \widehat{g_i}, [g_j, g^\vfi_k] \; \vert 
\; 1 \leq i \leq n, \, 1 \leq k < j \leq n \rangle$ 
is the subgroup $\cN^\prime_{n,2} \cN^q_{n,2}$ of $\cN_{n,2},$ while the subgroup 
$\langle [gj, g_i, g^\vfi_k] \rangle$, where 
$[gj, g_i, g_k]$ is a basic commutator of 
$\gamma_3(F_n) / \gamma_4(F_n),$ 
is isomorphic to 
$\gamma_3(F_n)/ \gamma_4(F_n) \gamma_3(F_n)^q \cong H_2(\cN_{n,2}, \mathbb{Z}_q)$,   
a homocyclic abelian group of exponent $q$ and $q-$rank $M_n(3) = \frac{1}{3}(n^3 - n).$ 
Consequently, also in this case we get that the sequence 
\eqref{eq:Schur2} splits and we then find that 
$$\cN_{n,2} \otimes^q \cN_{n,2} \cong 
C^{{n+1 \choose 2} + M_n(3)}_q \times 
\cN^\prime_{n,2} \cN^q_{n, 2}.$$
Therefore, 
$d(\cN_{n,2} \otimes^q \cN_{n,2}) = \frac{1}{3}n(n^2 + 3n + 2)$, thus showing that the upper bound given in 
Theorem~\ref{thm:Bacon} (ii) is also achieved when $q > 1.$ 
\end{proof}

\end{document}